\documentclass[12,reqno]{amsart}
\usepackage{amsmath, amssymb, amsthm}
\usepackage{url}
\usepackage[breaklinks]{hyperref}

\newcommand{\s}{{\sigma}}

 \renewcommand{\a}{\alpha}
\renewcommand{\b}{\beta}
\newcommand{\e}{\epsilon}

\renewcommand{\l}{\lambda}

\renewcommand{\t}{\tau}

\renewcommand{\(}{\left\(}
\renewcommand{\)}{\right\)}
\renewcommand{\[}{\left\[}
\renewcommand{\]}{\right\]}
\numberwithin{equation}{section}
 \theoremstyle{plain}
\newtheorem{theorem}{Theorem}[section]
\newtheorem{lemma}[theorem]{Lemma}

\newtheorem{corollary}[theorem]{Corollary}

   \makeatletter
\def\proof{\@ifnextchar[{\@oproof}{\@nproof}}
\def\@oproof[#1][#2]{\trivlist\item[\hskip\labelsep\textit{#2 Proof of\
#1.}~]\ignorespaces}
\def\@nproof{\trivlist\item[\hskip\labelsep\textit{Proof.}~]\ignorespaces}

\makeatother

\begin{document}
\title[On a certain divisor function in Number fields]{On a certain divisor function in Number fields} 

\author{ Rajat Gupta AND Sudip Pandit}\thanks{2010 \textit{Mathematics Subject Classification.} Primary  11R42; Secondary 11M41, 11N37.\\
\textit{Keywords and phrases.} Number fields, Dedekind zeta function, Divisor function, Riesz sum, Bessel functions.}
\address{Discipline of Mathematics, Indian Institute of Technology Gandhinagar, Palaj, Gandhinagar 382355, Gujarat, India} 
\email{rajat\_gupta@iitgn.ac.in, sudip.pandit@iitgn.ac.in }
\maketitle

\begin{abstract} The main aim of this paper is to study an analogue of the generalized divisor function in a number field $\mathbb{K}$, namely,  $\sigma_{\mathbb{K},\alpha}(n)$. The Dirichlet series associated to this function is $\zeta_{\mathbb{K}}(s)\zeta_{\mathbb{K}}(s-\alpha)$. We give an expression for the Riesz sum associated to $\s_{\mathbb{K},\a}(n),$ and also extend the validity of this formula by using convergence theorems. As a special case, when $\mathbb{K}=\mathbb{Q}$, the Riesz sum formula for the generalized divisor function is obtained, which, in turn, for $\a=0$, gives the Vorono\"{\i} summation formula associated to the divisor counting function $d(n)$. We also obtain a big $O$-estimate  for the Riesz sum associated to $\sigma_{\mathbb{K},\alpha}(n)$.  
\end{abstract}

\section{Introduction}
Divisor functions and its various generalizations are one of  the most important arithmetic  functions in number theory. The average order of the classical divisor function, studied by Dirichlet in 1849, is given by
\begin{equation}\label{average0}
\displaystyle{\sum_{n\leq x}}{\vphantom{\sum}}'d(n)=x\log x+(2\gamma -1)x+\dfrac{1}{4}+\Delta(x),
\end{equation}
where $d(n)$ denotes the number of divisors of $n$ and $\gamma$ is Euler's constant. Here, and in the sequel, the prime $'$ sign in the sum $\displaystyle{\sum_{n\leq x}}{\vphantom{\sum}}'a(n)$  will be used to indicate that whenever $x$ is an integer only $\dfrac{1}{2}a(x)$ is counted. Finding the exact order of magnitude of the error term $\Delta(x)$ as $x\rightarrow\infty$, known as \textit{Dirichlet divisor problem}, seems to be one of the most difficult problems in analytic number theory and is still unsolved to date. Till now the best estimate $\Delta(x)=O(x^{131/416+\epsilon})$, for any $\epsilon>0$, as $x\rightarrow \infty$, is due to Huxley \cite{huxley}. Hardy proved that $\Delta(x)\neq O(x^{1/4})$ and  it has been conjectured that $\Delta(x)=O(x^{1/4+\epsilon})$, for any $\epsilon>0$, as $x\rightarrow \infty$. 

Let $J_{\nu}(x)$, $Y_{\nu}(x)$ be the Bessel function of the first and second kind \cite[p. 40-64]{watson}, respectively, and $K_{\nu}(x)$ to be the modified Bessel function of the second kind \cite[p.~78]{watson}. Vorono\"{\i} in \cite{voronoi}, established a representation of $\Delta(x)$ in terms of Bessel functions, known as the \textit{Vorono\"{\i} summation formula}. It is given by
\begin{align}\label{voronoisum}
\displaystyle{\sum_{n\leq x}}{\vphantom{\sum}}'d(n)=x\log x+(2\gamma -1)x+\dfrac{1}{4}+\sum_{n=1}^{\infty}d(n)\left(\dfrac{x}{n}\right)^{1/2}I_{1}(4\pi \sqrt{nx}),
\end{align}
where 
\begin{align}\label{I1}
I_{\nu}(z):=-Y_{\nu}(z)-\dfrac{2}{\pi}K_{\nu}(z).
\end{align}

In $1931$ Dixon and Ferrar obtained an identity for the Riesz sum 
\begin{align*}
&\sum_{n\leq x}{\vphantom{\sum}}'d(n)(x-n)^k.
\end{align*}
However, their identity only valid for non-negative integers $k$. Oppenheim \cite{open}, and  Laurin\v{c}ikas \cite{Lauri} studied the Riesz sum associated to generalized divisor function $\s_{\a}(n):=\sum_{d|n}d^\a$. As a special case, Laurin\v{c}ikas also derive an extension of the Vorono\"{\i} summation formula \eqref{voronoisum}, namely,
\begin{align}\label{extendedvoro}
\sum_{n\leq x}&\s_{-\a}(n)=\zeta(1+\a)x+\frac{\zeta(1-\a)}{1-\a}x^{1-\a}-\frac{\zeta(\a)}{2}+\frac{x}{2\sin \left(\frac{\pi \a}{2} \right)}\sum_{n=1}^{\infty}\s_{\a}(n) \times\nonumber\\
&\times(\sqrt{nx})^{-1-\a}\left(J_{s-1}(4\pi \sqrt{nx})+J_{1-s}(4\pi \sqrt{nx}) -\frac{2}{\pi}\sin(\pi s)K_{1-s}(4\pi \sqrt{n x})\right).
\end{align}
Note that, Laurin\v{c}ikas proved this extended version of the Vorono\"{\i} summation formula for $0<\a<1/2$. In this paper, as a special case of our main theorem, we prove that this extended version is also valid in $|\textup{Re}(\a)|<1/2$ (see corollary \ref{Laurichikastheorem} below). 

This project stemmed from the study of the Risez sum associated to a certain divisor function in a number field. Let $\mathbb{K}$ be an algebraic number field of degree $\rho$ and $\mathcal{O}_{\mathbb{K}}$ be its ring of integers. Given an ideal $\mathfrak{a}\subset \mathcal{O}_{\mathbb{K}}$  the norm of the ideal is defined as $$\mathcal{N}(\mathfrak{a}):=|\mathcal{O}_{\mathbb{K}}/{\mathfrak{a}}|.$$



We define the generalized divisor function associated to a number field $\mathbb{K}$ by 
\begin{align}\label{divisordefinition}
\sigma_{\mathbb{K},\alpha}(n):=\sum_{d|n}a(d)a\left(\frac{n}{d}\right)d^\alpha,
\end{align}
where $a(d):=\#$ ideals in $\mathcal{O}_{\mathbb{K}}$ with norm $d$. 
It satisfy the following relation
\begin{align*}
\s_{\mathbb{K},\a}(n)n^{-\a/2}=\s_{\mathbb{K},-\a}(n)n^{\a/2}.
\end{align*} 

It is easy to see that the Dirichlet series associated to $\sigma_{\mathbb{K},\alpha}(n)$ is given by
\begin{align}
\sum_{n=1}^{\infty}\frac{\sigma_{\mathbb{K},\alpha}(n)}{n^s} =\zeta_{\mathbb{K}}(s)\zeta_{\mathbb{K}}(s-\alpha),~\mathrm{Re}(s)>\mathrm{max}\{1,1+\mathrm{Re}(\alpha)\},
\end{align}
where $\zeta_{\mathbb{K}}(s)$ is the Dedekind zeta function associated to the number field $\mathbb{K}$ defined by  
\begin{equation}\label{Dede_zeta_def}\zeta_{\mathbb{K}}(s):=\displaystyle{\sum_{\mathfrak{a\subset \mathcal{O}_{\mathbb{K}}}}\dfrac{1}{\mathcal{N}(\mathfrak{a})^{s}}}=\displaystyle{\sum_{n=1}^{\infty}\dfrac{a(n)}{n^{s}}} ,~\mathrm{for}~\mathrm{Re}(s)>1.
\end{equation} 
When $\mathbb{K}=\mathbb{Q}$, we have $a(d)=1,$ then $\sigma_{\mathbb{\mathbb{Q}},\alpha}(n)$ reduces to $\sigma_{\alpha}(n)=\sum_{d|n}d^{\alpha}$, and $\zeta_{\mathbb{\mathbb{Q}}}(s)$ is the Riemann zeta function $\zeta(s)$. 

If $r_1, ~r_2$ denote the number of real and complex (upto conjugate) embeddings of  $\mathbb{K}$, respectively, then we have
\begin{align}\label{r1r2}
r:=r_1+r_2-1 ~\mathrm{and}~ \rho:=r_1+2r_2,
\end{align}
where $r$ is the rank of $\mathcal{O}^{\times}_{\mathbb{K}}$ and $\rho$ is the degree of $\mathbb{K}/\mathbb{Q}$. Similar to the Riemann zeta function, $\zeta_{\mathbb{K}}(s)$ has a simple pole at $s=1$ with residue
\begin{align}
\lim_{s\to 1}(s-1)\zeta_{\mathbb{K}}(s)=\frac{2^{r+1}\pi^{r_2}R\cdot h}{\omega\sqrt{|\Delta_{\mathbb{K}}|}},
\end{align}
where $\omega$ denotes the number of roots of unity in $\mathbb{K}$, $R$ is the regulator, $h$ is the class number, and $\Delta_{\mathbb{K}}$ denotes the discriminant of $\mathbb{K}$.

To state our main theorem the following function plays an important role:
\begin{align}\label{DefJ}
x^{-\nu}\mathbb{J}_{\nu}(\a,x):=\frac{1}{2\pi i} \int_{c-i\infty}^{c+i\infty}\frac{2^{(2\rho s -\nu-\rho+1)}\Gamma(1-s)G(1-s)G(1-s+\a)}{\Gamma(\nu+1-s)}x^{-2s}ds
\end{align}
where 
\begin{align}\label{G(s)}
\max\{1,\textup{Re}(\a)\}<c<\frac{k+1+\rho(1+\textup{Re}(\a))}{2\rho}~\mathrm{and}~ G(s)=\frac{\Gamma^{r_1}\left(\frac{1-s}{2}\right)}{\Gamma^{r_1}\left(\frac{s}{2}\right)}\frac{\Gamma^{r_2}(1-s)}{\Gamma^{r_2}(s)}. 
\end{align}
Surprisingly, $\mathbb{J}_{\nu}(\a,x)$ enjoys a recurrence relation similar to that of $J_{\nu}(x)$ \cite[p. 202, Equation (4.6.1)]{aar}, that is, 
\begin{align}\label{diff}
\frac{d}{dx}\left[x^{\nu}\mathbb{J}_{\nu}(\a,x)\right]=x^{\nu}\mathbb{J}_{\nu-1}(\a,x).
\end{align}

We are now ready to state our first main theorem.
\begin{theorem}\label{maintheorem}
For $k>\rho\left(1+|\textup{Re}(\a)|\right)-\frac{1}{2}$ and $x>0,$
\begin{align}\label{maintheoremeqn}
&\frac{1}{\Gamma(k+1)}\sum_{n\leq x}{\vphantom{\sum}}'\sigma_{\mathbb{K},-\alpha}(n)(x-n)^k =R_{\a}(k,\rho,x)+\frac{2^{k(1-\rho)}}{A^{2(1+\a)}}\sum_{n=1}^{\infty}\s_{\mathbb{K},\a}(n)\left(\frac{A^4x}{n}\right)^{(1+k)/2}\mathbb{J}_{1+k}\left(\a,2^{\rho}\sqrt{\frac{nx}{A^4}}\right),
\end{align}
where $A= \displaystyle{\frac{\sqrt{|\Delta_{\mathbb{K}}|}}{2^{r_2}\pi^{\rho/2}}}$, and
\begin{align*}
R_{\a}(k,\rho,x):=\frac{\zeta_{\mathbb{K}}(0)\zeta_{\mathbb{K}}(\a)x^k}{k!}&+\frac{2^{r+1}\pi^{r_2}R\cdot h}{\omega\sqrt{|\Delta_{\mathbb{K}}|}}\frac{\zeta_{\mathbb{K}}(1+\a)}{(k+1)!}x^{k+1}\\
&\quad+\frac{2^{r+1}\pi^{r_2}R\cdot h}{\omega\sqrt{|\Delta_{\mathbb{K}}|}}\frac{\Gamma(1-\a)\zeta_{\mathbb{K}}(1-\a)}{\Gamma(2+k-\a)}x^{k+1-\a}.
\end{align*}
\end{theorem}

As an immediate application of the above theorem, we obtain a big $O$-estimate result for the left-hand side of \eqref{maintheoremeqn}.

\begin{corollary}\label{maintheoremO}
For $k>\rho\left(1+|\textup{Re}(\a)|\right)-\frac{1}{2}$, and $x>0$
\begin{align}\label{maintheoremOeqn}
&\frac{1}{\Gamma(k+1)}\sum_{n\leq x}{\vphantom{\sum}}'\sigma_{\mathbb{K},-\alpha}(n)(x-n)^k =R_{\a}(k,\rho,x)+O\left( x^{\frac{(2\rho-1)k+\rho\left(1-\textup{Re}(\a)\right)-1/2}{2\rho}}\right).
\end{align}
\end{corollary}

Next, we state our second main theorem, in which we extend the validity of \eqref{maintheoremeqn}. 

\begin{theorem}\label{extendedmaintheorem}
Theorem \ref{maintheorem} is valid for $k>\rho\left(1+|\textup{Re}(\a)|\right)-\frac{3}{2}.$
\end{theorem}
In Theorem \ref{extendedmaintheorem}, if we take  $\mathbb{K}=\mathbb{Q}$, then we obtain a formula given by Laurin\v{c}ikas \cite[Theorem 1]{Lauri} and \eqref{voronoisum}.
\begin{corollary}\label{Laurichikastheorem}
For $k>|\textup{Re}(\a)|-\frac{1}{2}$, 
\begin{align}
\frac{1}{\Gamma(k+1)}&\sum_{n\leq x}{\vphantom{\sum}}'\s_{-\a}(n)(x-n)^k =R_{\a}(k,1,x)+\frac{(2\pi)^{1+\a}x^{1+k}}{\sin\left(\frac{\pi \a}{2} \right)}\sum_{n=1}^{\infty}\s_{\a}(n)\l_{1+k}(4\pi\sqrt{nx},\a),
\end{align}
where 
\begin{align*}
\l_{\nu}(z,\a)&=\frac{1}{2}\left(\frac{z}{2} \right)^{-\nu-\a}\Bigg( J_{\nu-\a}(z,\a)+I_{\nu-\a}(z,\a)\Bigg)-\frac{1}{2}\left(\frac{z}{2} \right)^{-\nu}\Bigg( {}^+I_{\nu}(z,\a)- {}^+J_{\nu}(z,\a)\Bigg)
\end{align*}
defined in \cite[p.~82]{Lauri}. Hence, if we assume $k=0$, then we obtain \eqref{extendedvoro} for $|\textup{Re}(\a)|<\frac{1}{2}$. If we further assume $\a=0$, we get \eqref{voronoisum}.
\end{corollary}

\section{Nuts and Bolts}
We recall the Stirling bound in the vertical strip \cite[p. 21, Corollary 1.4.4]{aar}, namely, for $z=x+i y$, $x_1<x<x_2$  and $|y|\to \infty$,
\begin{align}\label{sterling}
|\Gamma(z)| = \sqrt{2\pi}~|y|^{x-1/2}e^{-\pi |y|/2}\left(1+O\left(\frac{1}{y} \right)\right).
\end{align} 

The Dedekind zeta function $\zeta_{\mathbb{K}}(s)$ admits the following functional equation: 
\begin{align}\label{functional}
\zeta_{\mathbb{K}}(s) =A^{1-2s}G(s)\zeta_{\mathbb{K}}(1-s),
\end{align}
where $A= \displaystyle{\frac{\sqrt{|\Delta_{\mathbb{K}}|}}{2^{r_2}\pi^{\rho/2}}}.$

Note that for $-\textup{Re}(\nu)<\t<3/2,$
\begin{align}\label{inversemellinBesselJ}
\frac{1}{2\pi i}\int_{(\t)}\frac{\Gamma((s+\nu)/2)}{2\Gamma((2-s+\nu)/2)}x^{-s}ds&=J_{\nu}(2x),
\end{align}
Here, and in the sequel, we abbreviate $\int_{\t-i\infty}^{\t+i\infty}$ by $\int_{(\t)}$.

We recall an important lemma from \cite[p. 5, Lemma 1]{chanar1}. 
\begin{lemma}\label{chandr}
Let $\sigma_a$ denote the abscissa of absolute convergence for 
\begin{align*}
\phi(s):=\sum_{n=1}^{\infty}\frac{a_n}{\lambda_n^s}.
\end{align*} 
Then for $k \geq 0,~\sigma >0,$ and $\sigma>\sigma_a$, 
\begin{align*}
\frac{1}{\Gamma(k+1)}\sum_{\lambda_n\leq x}{\vphantom{\sum}}'a_n(x-\lambda_n)^k =\frac{1}{2\pi i}\int_{(\sigma)}\frac{\Gamma(s)\phi(s)x^{s+k}}{\Gamma(s+k+1)}ds,
\end{align*}
where the prime $'$ on the summation denotes that if $k=0$ and $x=\lambda_n$ for some positive integer $m$, then we multiply by $1/2$ in $a_n$. 
\end{lemma}
We also take into consideration the following version of $\textup{Phragm\'{e}n-Lindel\"{o}f}$ principle \cite[p. 109]{lind}.
\begin{lemma}\label{lind}
Let $f$ be holomorphic in a strip $S$ given by $a<\sigma<b$, $|t|>\eta>0$, and continuous on the boundary. If for some constant $\theta<1$, 
\begin{align*}
f(s)\ll\textup{exp}\left(\theta\pi|s|/(b-a) \right),
\end{align*}
uniformly in $S$, $f(a+it)=\mathrm{o}(1),$ and $f(b+it)=\mathrm{o}(1),$ as $|t|\to \infty$, then 
\begin{align*}
f(\sigma+it) =\mathrm{o}(1)
\end{align*}
uniformly in $S$ as $|t|\to \infty$.
\end{lemma}

We conclude this section by recalling the asymptotic formula for $J_{\nu}(z)$ \cite[p. 199]{watson}, namely, 
\begin{align}\label{Besseljasym}
J_{\nu}(z)\sim\left(\frac{2}{\pi z} \right)^{1/2}\left( \cos w\sum_{n=0}^{\infty}\frac{(-1)^n(\nu,2n)}{(2z)^{2n}}-\sin w\sum_{n=0}^{\infty}\frac{(-1)^n(\nu,2n+1)}{(2z)^{2n+1}}\right),
\end{align}
where $w=z-\frac{1}{2}\pi\nu-\frac{\pi}{4}$.
\section{Proof of the main theorem}
We begin this section by deriving a crucial lemma, which shows that the behaviour of $\mathbb{J}_{\nu}(\a,x)$ is essentially depends on the nature of the Bessel function of the first kind $J_{\nu}(x)$.
\begin{lemma}\label{JJ}
For $x>0,$ and $k >\max\left\{\rho|1-\textup{Re}(\a)|-1/2,r_1-3/2\right\}$,
\begin{align*}
&\mathbb{J}_{1+k}(\a,2^{\rho}\sqrt{x})=2^{k(\rho-1)}x^{(k+1)/2}\Bigg( A^* x^{\frac{-k-\rho(1+\a)}{2\rho}}J_{k+\rho-2r-1}\left(2(xe^{-\Theta})^{\frac{1}{2\rho}}\right)+O\left( x^{\frac{-k-\rho(1+\textup{Re}(\a))-3/2}{2\rho}}\right)\Bigg).
\end{align*}
where
\begin{align*}
A^* &=-\frac{e^B\left(e^{-\Theta/2\rho} \right)^{-k-\rho(1+\a)}}{\rho},\\
B&=-r_1(\a+1)\log\frac{1}{2}-(\mu -\l)\log \rho,\\
\Theta&= 2r_1\log\frac{1}{2}-2\rho  \log \rho.
\end{align*}
\end{lemma}

\begin{proof}
From \eqref{DefJ}, we have, for $c>\max\{1,\textup{Re}(\a)\}$
\begin{align}\label{JJ}
x^{-\nu}\mathbb{J}_{\nu}(\a,x):=\frac{1}{2\pi i} \int_{c-i\infty}^{c+i\infty}\frac{2^{(2\rho s -\nu-\rho+1)}\Gamma(1-s)G(1-s)G(1-s+\a)}{\Gamma(\nu+1-s)}x^{-2s}ds.
\end{align}

Let's first define
\begin{align}\label{Ir1r2}
I_{r_1,r_2,k}(x):= \frac{1}{2\pi i}\int_{c-i\infty}^{c+i\infty}\frac{\Gamma(1-s)G(1-s)G(1-s+\a)}{\Gamma(2-s+k)}x^{-s}ds. 
\end{align}
Then we get
\begin{align}\label{JIrelation}
\mathbb{J}_{1+k}(\a,2^{\rho}\sqrt{x})=2^{k(\rho-1)}x^{(1+k)/2} I_{r_1,r_2,k}(x).
\end{align}

If $\mathbb{K}=\mathbb{Q}$, then  $\mathbb{J}_{\nu}(\a,x)$ reduces to a combination of Bessel functions $Y_{\nu}(x)$ and $K_{\nu}(x)$, and so with the help of the asymptotics of Bessel functions we can estimate its behaviour. On the other hand, the integrand is a quotient of multiple gamma factors. Therefore, in the case of general number fields we cannot always evaluate the line integral explicitly in terms of special functions. Hence, it's difficult to examine the behaviour of $\mathbb{J}_{\nu}(\a, x)$. However, the line integral with multiple gamma factors has been studied by Chandrasekharan and Narasimhan \cite{chanar2}. Hence the analogous ways and means can also be applied here, but now in a more general setting. To do so, let's consider, for any constant $c$, Stirling's approximation for $\log \Gamma(z)$ as $|z| \to \infty$, namely,
\begin{align}\label{sterlinglog}
\log \Gamma(z+c)=\left(z+c-\frac{1}{2}\right)\log z -z+\frac{1}{2}\log 2\pi +O\left( \frac{1}{|z|}\right).
\end{align}
Now we examine the asymptotic expansion of 
\begin{align}\label{Fas}
\log F_\a(s):=\log \frac{\Gamma(1-s)G(1-s)G(1-s+\a)}{\Gamma(2-s+k)},
\end{align}
using \eqref{sterlinglog}. As $|s|\to \infty$, we have
\begin{align}\label{sk}
\log\frac{\Gamma(1-s)}{\Gamma(2-s+k)}&=-(k+1)\log (-s)+O\left( \frac{1}{|s|}\right). 
\end{align}
From the definition \eqref{G(s)},
\begin{align}\label{1234}
\log G(1-s+\a)= \log \left(\frac{\Gamma^{r_1}\left(\frac{s-\a}{2}\right)}{\Gamma^{r_1}\left(\frac{1-s+\a}{2}\right)}\frac{\Gamma^{r_2}(s-\a)}{\Gamma^{r_2}(1-s+\a)}\right)=r_1\log \frac{\Gamma\left(\frac{s-\a}{2}\right)}{\Gamma\left(\frac{1-s+\a}{2}\right)}+r_2\log \frac{\Gamma(s-\a)}{\Gamma(1-s+\a)}. 
\end{align}
Employ \eqref{sterlinglog} on the right-hand side of \eqref{1234}, to see that, as $|s|\to \infty$,  
\begin{align*}
\log \frac{\Gamma\left(\frac{s-\a}{2}\right)}{\Gamma\left(\frac{1-s+\a}{2}\right)}&=\left(\frac{s-\a}{2}-\frac{1}{2} \right)\log\left(\frac{s}{2} \right)-\left(\frac{1-s+\a}{2}-\frac{1}{2} \right)\log\left(-\frac{s}{2} \right)-s+O\left( \frac{1}{|s|}\right)\\
&\qquad+\left(s-\a-\frac{1}{2} \right)\log\left(\frac{1}{2} \right)-s+O\left( \frac{1}{|s|}\right),
\end{align*}
and 
\begin{align*}
\log \frac{\Gamma(s-\a)}{\Gamma(1-s+\a)}&=\left(s-\a-\frac{1}{2} \right)\log\left(s\right)-\left(1-s+\a-\frac{1}{2} \right)\log\left(-s \right)-2s+O\left( \frac{1}{|s|}\right).
\end{align*}
Hence, we find
\begin{align}\label{1sa}
\log G(1-s+\a)&=r_1\log \frac{\Gamma\left(\frac{s-\a}{2}\right)}{\Gamma\left(\frac{1-s+\a}{2}\right)}+r_2\log \frac{\Gamma(s-\a)}{\Gamma(1-s+\a)}\nonumber\\
&=\left(\frac{\rho}{2}(s-\a)-\frac{r}{2}-\frac{1}{2} \right)\log\left(s \right)-\left(\frac{\rho}{2}(1-s+\a)-\frac{r}{2}-\frac{1}{2} \right)\log\left(-s \right)\nonumber\\
&\qquad+r_1\left(s-\a-\frac{1}{2} \right)\log\left(\frac{1}{2} \right)-\rho s+O\left( \frac{1}{|s|}\right); ~as~|s|\to \infty.
\end{align}
If we take $\a=0$ in \eqref{1sa}, we have, for $|s|\to \infty$,
\begin{align}\label{1s}
\log G(1-s)&=r_1\log \frac{\Gamma\left(\frac{s}{2}\right)}{\Gamma\left(\frac{1-s}{2}\right)}+r_2\log \frac{\Gamma(s)}{\Gamma(1-s)}\nonumber\\
&=\left(\frac{\rho}{2}s-\frac{r}{2}-\frac{1}{2} \right)\log\left(s \right)-\left(\frac{\rho}{2}(1-s)-\frac{r}{2}-\frac{1}{2} \right)\log\left(-s \right)\nonumber\\
&\qquad \qquad\qquad\qquad+r_1\left(s-\frac{1}{2} \right)\log\left(\frac{1}{2} \right)-\rho s+O\left( \frac{1}{|s|}\right).
\end{align}
Now from \eqref{Fas}, \eqref{sk}, \eqref{1sa}, and \eqref{1s}, for $|s|\to \infty$
\begin{align}
\log F_\a(s)
&=\left(\rho s+\mu -\frac{1}{2}\right)\log s -\left(\rho(1-s)+\mu +\rho \a +1+k-\frac{1}{2}\right)\log\left(-s \right)\nonumber\\
&\hspace{3cm}+2\nu s-r_1\left(\a+1\right)\log\left(\frac{1}{2} \right)+O\left( \frac{1}{|s|}\right),
\end{align}
where
\begin{align*}
\mu&=-\frac{\rho}{2}\a-r-\frac{1}{2},\\
\nu&=r_1\log \frac{1}{2}-\rho.
\end{align*}
If we define $\l:= \rho(1+\a)+\mu+1+k$, then
as $|s|\to \infty$, 
\begin{align}\label{asymAalpha}
\log F_\a(s)&=\left(\rho s+\mu -\frac{1}{2}\right)\log s -\left(-\rho s +\l-\frac{1}{2}\right)\log\left(-s \right)+2\nu s-r_1\left(\a+1\right)\log\left(\frac{1}{2} \right)+O\left( \frac{1}{|s|}\right).
\end{align}

Now let us consider the following function: 
\begin{align}\label{Halpha}
\log \frac{\Gamma(\rho s+\mu)}{\Gamma(\l -\rho s)}e^{\Theta s},
\end{align}
where 
\begin{align}
\l&= \rho(1+\a)+\mu+1+k,\\
\Theta&= 2\rho +2\nu-2\rho \log \rho.
\end{align}
Then by an application of the Stirling approximation \eqref{sterlinglog} in \eqref{Halpha}, and \eqref{asymAalpha}, we note that, as  $|s|\to \infty$,
\begin{align}\label{logfaplha}
\log F_\a(s)- \log \frac{\Gamma(\rho s+\mu)}{\Gamma(\l -\rho s)}e^{\Theta  s}= B +O\left( \frac{1}{|s|}\right),
\end{align}
where $$B=-r_1(\a+1)\log\frac{1}{2}-(\mu -\l)\log \rho.$$

If we define
\begin{align}\label{halphadef}
H_\a(s):=\frac{\Gamma(\rho s+\mu)}{\Gamma(\l -\rho s)}e^{\Theta s +B},
\end{align}
then, from \eqref{logfaplha} and \eqref{halphadef},
\begin{align}\label{FminusH}
F_\a(s)-H_\a(s)= H_{\a}(s)O\left( \frac{1}{|s|}\right).
\end{align}
Hence, 
\begin{align}\label{integralFalpha}
\frac{1}{2\pi i}\int_{(c)}F_{\a}(s)x^{-s}ds =\frac{1}{2\pi i}\int_{(c)}\left(F_{\a}(s)-H_{\a}(s)\right)x^{-s}ds +\frac{1}{2\pi i}\int_{(c)}H_{\a}(s)x^{-s}ds,
\end{align}
where, we take $c=\frac{k+1+\rho(1+\textup{Re}(\a))}{2\rho}-\e$, with $0<\e<\frac{1}{4\rho}$. Observe that we can choose such a $c$, only if $k>\rho|1-\textup{Re}(\a)|-1/2$.

Also note that the inequality, $-\frac{\textup{Re}(\mu)}{\rho}<c$ is true, provided $k>r_{1}-\frac{3}{2}$. Thus, the possible poles of the integrand are on the left side of the line, at $c.$
Now, by the analyticity of the integrand and \eqref{FminusH}, we can shift the integral by $1/2\rho$ to the right. Hence we have
\begin{align}\label{firstintegral}
\frac{1}{2\pi i}\int_{(c)}\left(F_{\a}(s)-H_{\a}(s)\right)x^{-s}ds&=\frac{1}{2\pi i}\int_{(c+\frac{1}{2\rho})}H_{\a}(s)O\left( \frac{1}{|s|}\right)x^{-s}ds\nonumber\\
&=O\left( x^{-\frac{k+1+\rho(1+\textup{Re}(\a))}{2\rho}+\e-\frac{1}{2\rho}}\right)\nonumber\\
&=O\left( x^{\frac{1}{2\rho}\left(-k-\rho(1+\textup{Re}(\a))-\frac{3}{2}\right)}\right),
\end{align}
and 
\begin{align*}
\int_{(c)}H_{\a}(s)x^{-s}ds&=\int_{(c)}\frac{\Gamma(\rho s+\mu)}{\Gamma(\l -\rho s)}e^{\Theta s +B}x^{-s}ds\\
&=\frac{e^{B}(xe^{-\Theta})^{\frac{\mu}{\rho}}}{\rho}\int_{(\rho c+\mu)}\frac{\Gamma(s)}{\Gamma(\l +\mu -s)}\left((xe^{-\Theta})^{\frac{1}{2\rho}}\right)^{-2s}ds.
\end{align*}
We now invoke the following version of \eqref{inversemellinBesselJ}:
\begin{align*}
\frac{1}{2\pi i}\int_{(\t)}\frac{\Gamma(s)}{\Gamma(\nu -s +1)}\left(x^2\right)^{-s}ds&=\frac{J_{\nu}(2x)}{x^{\nu}}, ~0<\t<3/4+\textup{Re}(\nu)/2
\end{align*}
to obtain
\begin{align}\label{secondintegral}
\frac{1}{2\pi i}\int_{(c)}H_{\a}(s)x^{-s}ds&=-\frac{e^{B}(xe^{-\Theta})^{\frac{\mu}{\rho}}}{\rho}\frac{1}{2\pi i}\int_{(\rho c+\mu)}\frac{\Gamma(s)}{\Gamma(\l +\mu -s)}\left((xe^{-\Theta})^{\frac{1}{2\rho}}\right)^{-2s}ds\nonumber\\
&=-\frac{e^{B}(xe^{-\Theta})^{\frac{\mu}{\rho}}}{\rho}\frac{J_{\l+\mu-1}\left(2(xe^{-\Theta})^{\frac{1}{2\rho}}\right)}{\left((xe^{-\Theta})^{\frac{1}{2\rho}}\right)^{\l+\mu-1}}\nonumber\\
&=-\frac{e^{B}\left((xe^{-\Theta})^{\frac{1}{2\rho}}\right)^{\mu-\l+1}}{\rho}J_{\l+\mu-1}\left(2(xe^{-\Theta})^{\frac{1}{2\rho}}\right)\nonumber\\
&=A^* x^{\frac{-k-\rho(1+\a)}{2\rho}}J_{k+\rho-2r-1}\left(2(xe^{-\Theta})^{\frac{1}{2\rho}}\right),
\end{align}
where in the last step we use the definition of $\l$, $\mu$, and $A^*$.

Note that the above evaluation is true only when
\begin{align*}
-\frac{\textup{Re}(\mu)}{\rho}<~ c~<\frac{1}{2\rho}\left(\frac{3}{2}+\rho(1+\textup{Re}(\a))+k\right),
\end{align*}
and our choice of $c$ in \eqref{integralFalpha} clearly satisfies the above condition.
Hence from \eqref{integralFalpha}, \eqref{firstintegral} and \eqref{secondintegral},
\begin{align}\label{Jbigo}
I_{r_1,r_2,k}(x):= \int_{c-i\infty}^{c+i\infty}&\frac{\Gamma(1-s)G(1-s)G(1-s+\a)}{\Gamma(2-s+k)}x^{-s}ds\nonumber\\
&=A^* x^{\frac{-k-\rho(1+\a)}{2\rho}}J_{k+\rho-2r-1}\left(2(xe^{-\Theta})^{\frac{1}{2\rho}}\right)+O\left( x^{\frac{1}{2\rho}\left(-k-\rho(1+\textup{Re}(\a))-\frac{3}{2}\right)}\right),
\end{align}
for $k >\max\left\{\rho|1-\textup{Re}(\a)|-1/2,r_1-3/2\right\}$. The lemma follows from the above equation and \eqref{JIrelation}.
\end{proof}

\begin{proof}[Theorem \textup{\ref{maintheorem}}][]
Let $a_n =\sigma_{\mathbb{K},-\alpha}(n)$ in Lemma \ref{chandr} to obtain 
\begin{align}
\frac{1}{\Gamma(k+1)}\sum_{n \leq x}{\vphantom{\sum}}'\sigma_{\mathbb{K},-\alpha}(n)(x-n)^k =\frac{1}{2\pi i}\int_{(\sigma)}\frac{\Gamma(s)\zeta_{\mathbb{K}}(s)\zeta_{\mathbb{K}}(s+\alpha)x^{s+k}}{\Gamma(s+k+1)}ds.
\end{align}
We define 
\begin{align}
\omega(s):=\frac{\Gamma(s)\zeta_{\mathbb{K}}(s)\zeta_{\mathbb{K}}(s+\alpha)x^{s+k}}{\Gamma(s+k+1)}.
\end{align}
Here we construct a contour defined by the line segment $[\s-iT, \s+iT]$, $[\s+iT, 1-\s+iT]$, $[1-\s+iT, 1-\s-iT]$, $[1-\s-iT, \s-iT]$, where $\s> \max\{1, 1\pm \textup{Re}(\a), \textup{Re}(\a)\}$, due to which $\omega(s)$ has three poles inside the contour, namely, a simple pole at zero of $\Gamma(s)$, a simple pole at $1$ of $\zeta_{\mathbb{K}}(s)$, and a simple pole at $s=1-\a$ due to $\zeta_{\mathbb{K}}(s+\a)$. Then by Cauchy's residue theorem,\begin{align}
\frac{1}{2\pi i}\left(\int_{\sigma-i T}^{\s+ i T}+\int_{\s+i T}^{1-\s +i T}\int_{1-\s +i T}^{1-\s -i T}+\int_{1-\s -i T}^{\s -i T} \right)\omega(s)ds=R_{0}+R_{1}+R_{1-\alpha},
\end{align}
where $R_{a}$ denote the residue of $\omega(s)$ at $s=a$, thus
\begin{align*}
R_{0}&=\lim_{s \to 0}s~\omega(s)=\frac{\zeta_{\mathbb{K}}(0)\zeta_{\mathbb{K}}(\a)x^k}{k!}, \\
R_{1}&=\lim_{s \to 1}(s-1)~\omega(s)=\frac{2^{r+1}\pi^{r_2}R\cdot h}{\omega\sqrt{|\Delta_{\mathbb{K}}|}}\frac{\zeta_{\mathbb{K}}(1+\a)}{(k+1)!}x^{k+1},\\
R_{1-\a}&=\lim_{s \to 1-\a}(s-1+\a)~\omega(s)=\frac{2^{r+1}\pi^{r_2}R\cdot h}{\omega\sqrt{|\Delta_{\mathbb{K}}|}}\frac{\Gamma(1-\a)\zeta_{\mathbb{K}}(1-\a)}{\Gamma(2+k-\a)}x^{k+1-\a}.
\end{align*}

Next we will show that the limits of the integrals along the horizontal lines are  zero, namely, 
\begin{align}
\int_{\sigma\pm iT}^{(1-\sigma)\pm iT}\omega(s)\ ds \to 0 ~ as ~|T|\to \infty.
\end{align}

If we take $s=c+iT$, where $1-\s\leq c \leq \s$, then we have, from $\eqref{sterling}$, 
\begin{align}\label{uniformgamma}
\frac{\Gamma(s)}{\Gamma(s+k+1)}= O\left(|T|^{-k-1}\right).
\end{align} 

At $s =\s \pm i T$, where $\s>\max\{1, 1\pm \textup{Re}(\a), \textup{Re}(\a)\}$, we know that $\zeta_{\mathbb{K}}(s)\zeta_{\mathbb{K}}(s+\a)$ is bounded, and along with \eqref{uniformgamma}, we obtain
\begin{align}\label{atsigma}
\omega(\s+iT)=\frac{\Gamma(s)\zeta_{\mathbb{K}}(s)\zeta_{\mathbb{K}}(s+\alpha)x^{s+k}}{\Gamma(s+k+1)} =\mathrm{o}(1),~ as~ |T|\to \infty.
\end{align}

But at $s =1-\s \pm iT$, we use the functional equation for $\zeta_{\mathbb{K}}(s)\zeta_{\mathbb{K}}(s+\a)$, namely, 
\begin{align}\label{functional2}
\zeta_{\mathbb{K}}(s)\zeta_{\mathbb{K}}(s+\a) =A^{2-4s-2\a}G(s)G(s+\a)\zeta_{\mathbb{K}}(1-s)\zeta_{\mathbb{K}}(1-s-\a).
\end{align}
Rademacher \cite{Rade} gave the following convexity bound for the Dedekind zeta function, namely,  as $|T|\to \infty$, 
\begin{align}\label{Rade}
\zeta_{\mathbb{K}}(\t+iT)  =
\left\{
	\begin{array}{ll}
		O\left(|T|^{\rho\left(\frac{1}{2}-\t\right)} \right)  & \t \leq 0 \\
		O\left(|T|^{\frac{\rho}{2}(1-\t)} \right)  & 0\leq \t \leq 1 \\
		O\left(1 \right) & \t\geq 1 \\
	\end{array}
\right.
\end{align}

Employing \eqref{uniformgamma}, and Stirling bound \eqref{sterling} in \eqref{atsigma}, we get 
\begin{align}
\omega(1-\s+i T) =\frac{\Gamma(s)\zeta_{\mathbb{K}}(s)\zeta_{\mathbb{K}}(s+\alpha)x^{s+k}}{\Gamma(s+k+1)} =O\left(|T|^{\rho(2\s -\textup{Re}(\a) -1)-k-1} \right)=\mathrm{o}(1),~ as~ |T|\to \infty, 
\end{align}
when $k>\rho(2\s -\textup{Re}(\a) -1) -1.$

From \eqref{Rade}, we get a convexity bound for Dedekind zeta function in the strip $1-\s\leq \t \leq \s$, namely, 
\begin{align}\label{Rade4}
\zeta_{\mathbb{K}}(\t+i T) = O \left( |T|^{\frac{\rho}{2}\left( \s -\t\right)}\right).
\end{align}
We invoke Lemma \ref{lind} and use \eqref{uniformgamma} and \eqref{Rade4} in the strip  $1-\s\leq \t \leq \s$, to obtain 
\begin{align}
\omega(s)=\frac{\Gamma(s)\zeta_{\mathbb{K}}(s)\zeta_{\mathbb{K}}(s+\alpha)x^{s+k}}{\Gamma(s+k+1)} =O\left(\textup{exp}\left(C|s|\log|s| \right) \right),
\end{align}
for some constant $C$ as $|T|\to \infty.$ Hereby, we deduce that 
\begin{align*}
\frac{\Gamma(s)\zeta_{\mathbb{K}}(s)\zeta_{\mathbb{K}}(s+\alpha)x^{s+k}}{\Gamma(s+k+1)} =\mathrm{o}(1),
\end{align*}
uniformly for $1-\s\leq\t\leq\s$ and $|T|\to \infty.$
Hence, by applying Lemma \ref{lind}, we get
\begin{align}
\int_{\sigma\pm iT}^{(1-\sigma)\pm iT}\omega(s)\ ds \to 0, ~ as ~|T|\to \infty.
\end{align}
Thus we have shown that for $k>\rho(2\s -\textup{Re}(\a) -1) -1$,
\begin{align}\label{finalcauchy}
\frac{1}{2\pi i}\int_{\sigma-i \infty}^{\s+i \infty}\omega(s)ds=R_{0}+R_{1}+R_{1-\alpha}+\frac{1}{2\pi i}\int_{1-\sigma-i\infty}^{1-\s+i\infty}\omega(s)ds.
\end{align}

Now we turn our attention towards computing the line integral on the right-hand side of \eqref{finalcauchy}. To do that, we use functional equation of $\zeta_{\mathbb{K}}(s)\zeta_{\mathbb{K}}(s+\alpha)$ from \eqref{functional2} to get
\begin{align*}
\int_{1-\sigma-i\infty}^{1-\s+i\infty}\omega(s)ds&= \int_{1-\sigma-i\infty}^{1-\s+i\infty}\frac{\Gamma(s)\zeta_{\mathbb{K}}(s)\zeta_{\mathbb{K}}(s+\alpha)x^{s+k}}{\Gamma(s+k+1)}\ ds \\
&=A^{2(1-\a)}x^k\int_{1-\sigma-i\infty}^{1-\s+i\infty}\frac{\Gamma(s)G(s)G(s+\a)\zeta_{\mathbb{K}}(1-s)\zeta_{\mathbb{K}}(1-s-\a)}{\Gamma(s+k+1)}\left(\frac{x}{A^4}\right)^{s}\ ds \\
&=A^{2(1-\a)}x^k\sum_{n=1}^{\infty}\frac{\s_{\mathbb{K},\a}(n)}{n}\int_{1-\sigma-i\infty}^{1-\s+i\infty}\frac{\Gamma(s)G(s)G(s+\a)}{\Gamma(s+k+1)}\left(\frac{nx}{A^4}\right)^{s}\ ds.
\end{align*}
Upon employing the change of variable $s \to 1-s$, we obtain, for $k>\rho(2\s -\textup{Re}(\a) -1) -1$,
\begin{align}
\int_{1-\sigma-i\infty}^{1-\s+i\infty}\omega(s)ds&=\frac{x^{1+k}}{A^{2(1+\a)}}\sum_{n=1}^{\infty}\s_{\mathbb{K},\a}(n)\int_{\sigma-i\infty}^{\s+i\infty}\frac{\Gamma(1-s)G(1-s)G(1-s+\a)}{\Gamma(2-s+k)}\left(\frac{nx}{A^4}\right)^{-s}ds \nonumber\\
&=\frac{2^{(k+\rho)}x^{1+k}}{A^{2(1+\a)}}\sum_{n=1}^{\infty}\s_{\mathbb{K},\a}(n)\left(2^{\rho}\sqrt{\frac{nx}{A^4}}\right)^{-(1+k)}\mathbb{J}_{1+k}\left(\a,2^{\rho}\sqrt{\frac{nx}{A^4}}\right)\nonumber\\
&=\frac{2^{k(1-\rho)}}{A^{2(1+\a)}}\sum_{n=1}^{\infty}\s_{\mathbb{K},\a}(n)\left(\frac{A^4x}{n}\right)^{(1+k)/2}\mathbb{J}_{1+k}\left(\a,2^{\rho}\sqrt{\frac{nx}{A^4}}\right), 
\end{align}
where $\mathbb{J}_{\nu}(\a,x)$ is defined in $\eqref{DefJ}.$
Now with the help of Lemma \ref{JJ} we find the region of absolute convergence of the series 
\begin{align}\label{series}
\sum_{n=1}^{\infty}\s_{\mathbb{K},\a}(n)\left(\frac{A^4x}{n}\right)^{(1+k)/2}\mathbb{J}_{1+k}\left(\a,2^{\rho}\sqrt{\frac{nx}{A^4}}\right).
\end{align}
From Lemma \ref{JJ} and \eqref{Besseljasym} we know that
\begin{align*}
&\mathbb{J}_{1+k}\left(\a,2^{\rho}\sqrt{\frac{nx}{A^4}}\right)=O\left(\left(\frac{n x}{A^4}\right)^{\frac{2k(\rho-1)-2\rho\textup{Re}(\a)-1}{4\rho}}\right),
\end{align*}
provided $k>\max\left\{\rho|1-\textup{Re}(\a)|-1/2,r_1-3/2\right\}$. Thus, for $x>0$, we have
\begin{align}\label{O}
\left|\sum_{n=1}^{\infty}\s_{\mathbb{K},\a}(n)\left(\frac{A^4x}{n}\right)^{(1+k)/2}\mathbb{J}_{1+k}\left(\a,2^{\rho}\sqrt{\frac{nx}{A^4}}\right) \right|&\ll_{\rho,A} x^{\frac{(2\rho-1)k+\rho\left(1-\textup{Re}(\a)\right)-1/2}{2\rho}}\sum_{n=1}^{\infty}\frac{\s_{\mathbb{K},\textup{Re}(\a)}(n)}{n^{\frac{k+\rho(1+\textup{Re}(\a))+1/2}{2\rho}}}\nonumber\\
&\ll_{\rho,A} \left( x^{\frac{(2\rho-1)k+\rho\left(1-\textup{Re}(\a)\right)-1/2}{2\rho}}\right),
\end{align}
whenever $k>\rho\left(1+|\textup{Re}(\a)|\right)-\frac{1}{2}$.  This implies that the series on the right-hand side of Theorem \ref{maintheorem} is absolutely and uniformly convergent in $k>\rho\left(1+|\textup{Re}(\a)|\right)-\frac{1}{2}>\max\left\{\rho|1-\textup{Re}(\a)|-1/2,r_1-3/2\right\}$. In view of \eqref{diff} and after a suitable number of differentiations, we can confirm the validity of \eqref{maintheoremeqn} for  $k>\rho\left(1+|\textup{Re}(\a)|\right)-\frac{1}{2}>0.$ 
\end{proof}

\begin{proof}[Corollary \textup{\ref{maintheoremO}}][]
Proof is immediate from \eqref{O}.
\end{proof}

\section{Convergence Theorem}
In this section we wish to extend the region of validity of Theorem \ref{maintheorem}. The method we adopt here uses the convergence results given in Chandrasekharan and Narasimhan \cite{chanar1}, Dixit et al. \cite{DB},  and Berndt \cite{berndt}, which, in turn, are consequences of some results of Zygmund \cite{Zygmund}, \cite[Lemma 8]{chanar1}.

Here we note that  $\sum_{n=-\infty}^{\infty}a_n(x)$ and $\sum_{n=-\infty}^{\infty}b_n(x)$ are \textit{uniformly equi-convergent} on an interval, provided the sum $\sum_{j=-n}^{n}[a_j(x)-b_{j}(x)]$ converges uniformly on that interval as $n \to \infty.$

Next, we recall two lemmas from \cite[Corollary 1 and 2 ]{chanar1}. 
\begin{lemma}\label{1}
Let $a_n$ be a positive strictly increasing sequence of numbers tending to $\infty,$ and assume $a_n =a_{-n}$.  Suppose that $J$ is a closed interval contained in an interval $I$ of length $2\pi$. Assume that 
\begin{align*}
\sum_{n=-\infty}^{\infty}|c(n)|<\infty.
\end{align*} 
Then, If $g$ is a function with period $2\pi$ which equals 
\begin{align*}
\sum_{n=-\infty}^{\infty}c(n)e^{ia_n x}
\end{align*}
on $I$, the Fourier series of $g$ converges uniformly on $J.$
\end{lemma}

\begin{lemma}\label{2}
With the same notation as above in Lemma \ref{1}, assume that 
\begin{align*}
\sup_{0\leq h\leq 1}\left|\sum_{k\leq a_n \leq k+h}c_n \right|=\mathrm{o}(1),
\end{align*}
as $k \to \infty$, and
\begin{align*}
\sum_{n=-\infty}^{\infty}\frac{|c_n|}{a_n}<\infty.
\end{align*}
Let $A(x)$ be a $C^{\infty}$ function with compact support on $I$, which is equal to $1$ on $J.$ Furthernore, let $B(x)$ be a $C^{\infty}$ function. Then, the series\begin{align*}
B(x)\sum_{n=-\infty}^{\infty}c(n)e^{ia_n x}
\end{align*}
is uniformly equi-convergent on $J$ with the differentiated series of the Fourier series of function with period $2\pi$, which equals 
\begin{align*}
A(x)\sum_{n=-\infty}^{\infty}c(n)W_n(x)
\end{align*}
on $I$, where $W_n(x)$ is an antiderivative $B(x)e^{ia_n x}$.
\end{lemma}

\begin{lemma}\label{k+2wala}
Suppose that
\begin{align}\label{3/2}
\sum_{n=1}^{\infty}\frac{|\s_{\mathbb{K},\a}(n)|}{n^{\frac{k+\rho\left(1+\textup{Re}(\a)\right)+3/2}{2\rho}}}< \infty, 
\end{align}
and
\begin{align}\label{1/2}
\sup_{0\leq h\leq 1}\left|\sum_{m^{2\rho}\leq \mu_n \leq (m+h)^{2\rho}} \frac{|\s_{\mathbb{K},\a}(n)|}{n^{\frac{k+\rho\left(1+\textup{Re}(\a)\right)+1/2}{2\rho}}}\right|=\mathrm{o}(1)
\end{align}
as $m \to \infty$.

Define, for $y>0$
\begin{align*}
F_{k+1}(\a,y):=\sum_{n=1}^{\infty}\s_{\mathbb{K},\a}(n)\left(\frac{A^4y^{2\rho}}{n}\right)^{(1+k)/2}\mathbb{J}_{1+k}\left(\a,(2y)^{\rho}\sqrt{\frac{n}{A^4}}\right),
\end{align*}
Then, $\rho\cdot 2^{\rho}y^{2\rho-1}F_{k+1}(\a,y)$ is uniformly equi-convergent on any interval $J$ of length less than $\pi/\rho$ with the differentiated series of the Fourier series of a function with period $\pi/\rho$ which on $I$ equals $A(y)F_{k+2}(\a,y)$, where $I$ is of length $\pi/\rho$ and contains $J.$
\end{lemma}
\begin{proof}
Consider the function, 
\begin{align*}
f(y):=\rho\cdot &2^{\rho}y^{2\rho-1+\rho(k+1)}\sum_{n=1}^{\infty}\s_{\mathbb{K},\a}(n)\left(\frac{A^4}{n}\right)^{(1+k)/2}\Bigg\{\mathbb{J}_{1+k}\left(\a,(2y)^{\rho}\sqrt{\frac{n}{A^4}}\right)\\
&-A^*\sqrt{\frac{e^{\frac{\Theta}{2\rho}}}{\pi}} \left(\frac{n y^{2\rho}}{A^4}\right)^{\frac{2k(\rho-1)-2\rho\a-1}{4\rho}}\cos w'+c_0A^*\sqrt{\frac{e^{\frac{\Theta}{2\rho}}}{\pi}}\left(\frac{n  y^{2\rho}}{A^4}\right)^{\frac{2k(\rho-1)-2\rho\a-3}{4\rho}}\sin w'\Bigg\}
\end{align*}
Upon invoking \eqref{Besseljasym}, \eqref{JJ}, and \eqref{3/2}, we see that $f(y)$ is continuously differentiable for $y>0.$ Assume $g(x)$ to be a function with period $\frac{\pi}{\rho}$ which coincides with $f(y)$ in interval $I$. Since $f(y)$ is continuously differentiable, the Fourier series expansion of $g(y)$ is uniformly convergent on interval $J$. By the hypothesis \eqref{3/2}, \eqref{1/2}, and Lemma \ref{2}, the series 
\begin{align*}
2^{\rho}A^*\sqrt{\frac{e^{\frac{\Theta}{2\rho}}}{\pi}}y^{2\rho-1}\sum_{n=1}^{\infty}\s_{\mathbb{K},\a}(n)\left(\frac{A^4y^{2\rho}}{n}\right)^{(1+k)/2} \left(\frac{n y^{2\rho}}{A^4}\right)^{\frac{2k(\rho-1)-2\rho\a-1}{4\rho}}\cos w'
\end{align*}
is uniformly equi-convergent on $J$ with the differentiated series of the Fourier series of a function with period $\frac{\pi}{\rho}$, which equals on $I$, given by
\begin{align*}
2^{\rho}A^*\sqrt{\frac{e^{\frac{\Theta}{2\rho}}}{\pi}}A(y)\sum_{n=1}^{\infty}\s_{\mathbb{K},\a}(n)\int_{\b}^{y}\left(\frac{A^4t^{2\rho}}{n}\right)^{(1+k)/2} \left(\frac{n t^{2\rho}}{A^4}\right)^{\frac{2k(\rho-1)-2\rho\a-1}{4\rho}}\cos w'dt, 
\end{align*}
where $\b>0.$ Using Lemma \ref{1}, we can derive a similar result for series
\begin{align*}
2^{\rho} c_0A^*\sqrt{\frac{e^{\frac{\Theta}{2\rho}}}{\pi}}y^{2\rho-1}\sum_{n=1}^{\infty}\s_{\mathbb{K},\a}(n)\left(\frac{A^4y^{2\rho}}{n}\right)^{(1+k)/2}\left(\frac{n  y^{2\rho}}{A^4}\right)^{\frac{2k(\rho-1)-2\rho\a-3}{4\rho}}\sin w'.
\end{align*}
Hence,  
\begin{align*}
2^{\rho}y^{2\rho-1}F_{k+1}(\a,y)=2^{\rho}y^{2\rho-1}\sum_{n=1}^{\infty}\s_{\mathbb{K},\a}(n)\left(\frac{A^4y^{2\rho}}{n}\right)^{(1+k)/2}\mathbb{J}_{1+k}\left(\a,(2y)^{\rho}\sqrt{\frac{n}{A^4}}\right)
\end{align*}
is uniformly equi-convergent on $J$ with the differentiated series of Fourier series of a function with period $\pi/\rho$, which equals on $I$, thus, if we invoke \eqref{diff} and choose $\b$ such that 
\begin{align*}
\lim_{y \to \b}y^{(k+2)\rho}\cdot\mathbb{J}_{1+k}\left(\a,(2y)^{\rho}\sqrt{\frac{n}{A^4}}\right) \to 0,
\end{align*}
then
\begin{align*}
A(y)\sum_{n=1}^{\infty}\s_{\mathbb{K},\a}(n)\int_{\b}^{y}\rho\cdot  2^{\rho}t^{2\rho-1}&\left(\frac{A^4t^{2\rho}}{n}\right)^{(1+k)/2}\mathbb{J}_{1+k}\left(\a,(2t)^{\rho}\sqrt{\frac{n}{A^4}}\right)dt\\
&=A(y)F_{k+2}(\a,y).
\end{align*}
This completes the proof.
\end{proof}

\begin{proof}[Theorem \textup{\ref{extendedmaintheorem}}][]
Let's assume, for $\b>\max\{1,1+\textup{Re}(\a) \}$\begin{align*}
\sum_{n=1}^{\infty}\frac{|\s_{\mathbb{K},\a} (n)|}{n^{\b}}<\infty,
\end{align*}
and 
\begin{align*}
\sup_{0\leq h\leq 1}\left|\sum_{m^{2\rho}\leq \mu_n \leq (m+h)^{2\rho}} \frac{|\s_{\mathbb{K},\a}(n)|}{n^{\b-\frac{1}{2\rho}}}\right|=\mathrm{o}(1); ~as~m \to \infty.
\end{align*}
Put $x=y^{2\rho}$ in \eqref{maintheoremeqn}, where $y$ is in the interval of length less than $\frac{\pi}{\rho}$. Then,  Lemma \ref{k+2wala} holds true, provided $k>\rho\left(1+|\textup{Re}(\a)|\right)-\frac{3}{2}$. That is, $k+1>\rho\left(1+|\textup{Re}(\a)|\right)-\frac{1}{2}$, hence, Theorem \ref{maintheorem} implies,
\begin{align*}
A(y)&\left(\frac{1}{\Gamma(k+2)}\int_{0}^{y}\sum_{n\leq t^{2\rho}}{\vphantom{\sum}}'2\rho\cdot\sigma_{\mathbb{K},-\alpha}(n)(t^{2\rho}-n)^kt^{2\rho-1} - R_{\a}\left(k+1,\rho,y^{2\rho}\right)\right)\\
&=A(y)\frac{2^{(k+1)(1-\rho)}}{A^{2(1+\a)}}\sum_{n=1}^{\infty}\s_{\mathbb{K},\a}(n)\left(\frac{A^4y^{2\rho}}{n}\right)^{(2+k)/2}\mathbb{J}_{2+k}\left(\a,2^{\rho}\sqrt{\frac{nx}{A^4}}\right)\\
&=A(y)\frac{2^{(k+1)(1-\rho)}}{A^{2(1+\a)}}F_{k+2}(\a,y). 
\end{align*}
Since $A(y)=1$ on interval $J$, upon invoking the localization principle \cite[Theorem 6.6, p. 53]{Zygmund2} and the properties of the Fourier series
\begin{align*}
\frac{2\rho \cdot t^{2\rho-1}}{\Gamma(k+2)}\sum_{n\leq t^{2\rho}}{\vphantom{\sum}}'\sigma_{\mathbb{K},-\alpha}(n)(t^{2\rho}-n)^k
\end{align*} 
in interval $I$, we see that, \eqref{maintheoremeqn} holds true for $k>\rho\left(1+|\textup{Re}(\a)|\right)-\frac{3}{2}.$
\end{proof}

Since we derive Theorem \ref{maintheorem} for extended region $k>\rho\left(1+|\textup{Re}(\a)|\right)-\frac{3}{2}$, we obtain a result of Laurin\v{c}ikas \cite[Theorem]{Lauri} for rational field in region $k>|\textup{Re}(\a)|-1/2$.  
\begin{proof}[Corollary \textup{\ref{Laurichikastheorem}}][]
Let $\mathbb{K}=\mathbb{Q}$ in Theorem \ref{extendedmaintheorem}, we have for $k>|\textup{Re}(\a)|-\frac{1}{2}$ and $x>0,$
\begin{align}
\frac{1}{\Gamma(k+1)}&\sum_{n\leq x}{\vphantom{\sum}}'\sigma_{\mathbb{K},-\alpha}(n)(x-n)^k \nonumber \\
&=R_{\a}(k,\rho,x)+\frac{2^{k(1-\rho)}}{A^{2(1+\a)}}\sum_{n=1}^{\infty}\s_{\mathbb{K},\a}(n)\left(\frac{A^4x}{n}\right)^{(1+k)/2}\mathbb{J}_{1+k}\left(\a,2^{\rho}\sqrt{\frac{nx}{A^4}}\right).
\end{align}

Since $\mathbb{K}=\mathbb{Q}$, we note 
$$\rho=1,~A=\frac{1}{\sqrt{\pi}},~and~\sigma_{\mathbb{K},-\alpha}(n)=\s_{-\a}(n)=\sum_{d|n}d^{-\a}.$$ Then
from the above equation we obtain
\begin{align}\label{4.4}
\frac{1}{\Gamma(k+1)}&\sum_{n\leq x}{\vphantom{\sum}}'\s_{-\a}(n)(x-n)^k \nonumber \\
&=R_{\a}(k,1,x)+\pi^{(1+\a)}\sum_{n=1}^{\infty}\s_{\a}(n)\left(\frac{x}{n\pi^2}\right)^{(1+k)/2}\mathbb{J}_{1+k}\left(\a,2\pi\sqrt{nx}\right).
\end{align}
Upon using the definition \eqref{JJ} we can rewrite the sum on the right-hand side of the above equation as 
\begin{align}\label{sumJI}
&\pi^{(1+\a)}\sum_{n=1}^{\infty}\s_{\a}(n)\left(\frac{x}{n\pi^2}\right)^{(1+k)/2}\mathbb{J}_{1+k}\left(\a,2\pi\sqrt{nx}\right)=\pi^{(1+\a)}x^{1+k}\sum_{n=1}^{\infty}\s_{\a}(n)I_{1,0,x}(\pi^2 nx),
\end{align}
where from \eqref{G(s)} and \eqref{Ir1r2} 
\begin{align*}
I_{1,0,x}(x)=\frac{1}{2\pi i}\int_{\s-i\infty}^{\s+i\infty}\frac{\Gamma(1-s)\Gamma\left(\frac{s}{2} \right)\Gamma\left(\frac{s-\a}{2} \right)}{\Gamma(2-s+k)\Gamma\left(\frac{1-s}{2} \right)\Gamma\left(\frac{1-s+\a}{2} \right)}x^{-s}ds; \s>1.
\end{align*}
From here, we can use \cite[Equation 9 and Lemma 3]{Lauri} to complete the proof.
\end{proof}

\begin{center}
Acknowledgement 
\end{center}

We would like to thank Prof. Atul Dixit and Prof. Arnab Saha for their valuable comments and constant support throughout the research. We would also like to take this opportunity to thank our institution Indian Institute of Technology Gandhinagar for providing us a state-of-the-art research facilities.

\bibliographystyle{amsalpha}

\end{document}